\documentclass[11pt]{iopart}
\usepackage{iopams}
\usepackage{amssymb}
\usepackage{algorithm}
\usepackage{algorithmic}

\usepackage{tikz}
\usetikzlibrary{fpu}
\usepackage{pgfplots}
\usepackage{todonotes}

\newtheorem{theorem}{Theorem}[section]
\newtheorem{lemma}[theorem]{Lemma}
\newtheorem{corollary}[theorem]{Corollary}
\newtheorem{example}[theorem]{Example}
\newenvironment{proof}{{\bf Proof}. \ignorespaces}{$\square$\par}
\begin{document}

\title[Rational Krylov subspace regularisation]{A conjugate-gradient-type rational Krylov subspace method for ill-posed problems}

\author{Volker Grimm}

\address{
        Institute for Applied and Numerical Mathematics,
        Karlsruhe Institute of Technology,
        D--76128 Karlsruhe,
        Germany}
\ead{volker.grimm@kit.edu}
\vspace{10pt}
\begin{indented}
\item[]August 2019
\end{indented}

\begin{abstract}
Conjugated gradients on the normal equation (CGNE) is a popular method to regularise linear inverse problems. The idea of the method can be summarised as minimising the residuum over a suitable Krylov subspace. It is shown that using the same idea for the shift-and-invert rational Krylov subspace yields an order-optimal regularisation scheme.  
\end{abstract}

\section{Introduction}
We consider the solution of the linear system 
\begin{equation} \label{lgs}
 Tx=y^\delta
\end{equation}
where the operator $T$ acts continuously between the Hilbert spaces $\mathcal{X}$ and $\mathcal{Y}$. 
The linear system is assumed to be ill-posed, that is, the range $\mathcal{R}(T)$ is not closed in $\mathcal{Y}$. $y^\delta$ is a perturbation of the exact data $y$, such that $\|y^\delta-y\| \leq \delta$. $y^\delta$ is also called the {\em noisy data} and $\delta$ the {\em noise level}. For exact data, we assume that $y$ is in the range of $T$, $y \in \mathcal{R}(T)$, which guarantees that there exists a unique $x^+ \in \mathcal{N}(T)^\perp$ such that $Tx^+=y$. $\mathcal{N}(T)^\perp$ designates the orthogonal complement of the null space $\mathcal{N}(T)$ of $T$. $x^+$
can also be characterised as the unique $x^+ \in \mathcal{N}(T)^\perp$ that solves the {\em normal equation}
\begin{equation} \label{normalequation}
T^*T x=T^*y\,.
\end{equation}
In fact, the normal equation
possesses a unique solution $x^+ \in \mathcal{N}(T)^\perp$ for every $y \in \mathcal{D}(T^+):=\mathcal{R}(T) \oplus \mathcal{R}(T)^\perp$, where $\mathcal{R}(T) \oplus \mathcal{R}(T)^\perp$ designates the direct orthogonal sum of $\mathcal{R}(T)$ and its orthogonal complement $\mathcal{R}(T)^\perp=\mathcal{N}(T^*)$.
The linear unbounded map $T^+:\mathcal{D}(T^+) \rightarrow \mathcal{N}(T)^\perp$, $y \mapsto x^+$, is the {\em Moore--Penrose (generalised) inverse} and $x^+$ is the {\em minimum-norm} solution.

In order to reconstruct the solution $x^+$ of the unperturbed problem $Tx^+=y$ as good as possible subject to a given noise level $\delta$, special procedures, called {\em regularisation schemes}, have to be used. Let $\{R_m\}_{m \in \mathbb{N}_0}$ be a family of linear or nonlinear operators from $\mathcal{Y}$ to $\mathcal{X}$ with $R_m0=0$. If there exists a mapping $m : \mathbb{R}^+ \times \mathcal{Y} \rightarrow \mathbb{N}_0$ such that
\[
 \limsup_{\delta \rightarrow 0} \{ \|R_{m(\delta,y^\delta)}y^\delta-x^+\|~|~ y^\delta \in \mathcal{Y}, \|y^\delta-Tx^+\|\leq \delta\}=0
\]
for any $x^+ \in \mathcal{N}(T)^\perp$, then the pair $(R_m,m(\delta,y^\delta))$ is called a (convergent) regularisation scheme for $T$. The mapping $m$ is called {\em parameter choice} or {\em stopping rule}. We will always use the {\em discrepancy principle} as our stopping rule, which is due to Morozov \cite{Morozov66}. The {\em discrepancy principle} reads: 
Choose a fixed $\tau > 1$ and set:
\begin{equation} \label{discprin}
  m(\delta,y^\delta) := \min\{m \in \mathbb{N}_0~|~\|y^\delta-Tx_m^\delta\| \leq \tau \delta \}\,,
\end{equation}
where $x_m^\delta := R_my^\delta$. The discrepancy principle leads to convergent regularisation schemes (cf. \cite{EHaNeu96}).
Regularisation schemes might converge arbitrarily slow unless the (unperturbed) data $x^+$ satisfies some smoothness assumptions. Convergence rates can be given when $x^+$ is in the {\em source set} 
$\mathcal{X}_{\mu,\rho}:=\{ x\in \mathcal{X}~|~x=(T^*T)^\mu w,~\|w\| \leq \rho\}$, $\mu > 0$. Regularisation schemes $(R_m,m(\delta, y^\delta))$ that attain the highest possible convergence speed are called of {\em optimal order} in $\mathcal{X}_{\mu,\rho}$ if 
\[
\sup\{ \|R_{m(\delta,y^\delta)}y^\delta-x^+\|~|~x^+\in \mathcal{X}_{\mu,\rho}, \|y^\delta-T x^+\| \leq \delta \} \leq C_{\mu}\delta^{\frac{2\mu}{2\mu+1}}\rho^{\frac{1}{2\mu+1}}\,,
\]
where $C_\mu$ neither depends on $\delta$ nor on $\rho$.

One of the most popular iterative regularisation schemes is {\em Conjugated Gradients on the Normal Equation} {\sc (CGNE)} that can be stated briefly as
\begin{equation} \label{minpropCGNE}
  x_m^\delta =: R_my^\delta, \qquad x_m^\delta = \mbox{argmin}_{x \in \mathcal{K}_m} \|y^\delta - Tx\|, \qquad m=1,2,\ldots \,,
\end{equation}
where $\mathcal{K}_m$ is the (polynomial) Krylov subspace
\[
  \mathcal{K}_m=\mathcal{K}_m(T^*T,T^*y^\delta)=\mbox{span}\{T^*y^\delta, (T^*T)T^*y^\delta,\ldots,
 (T^*T)^{m-1}T^*y^\delta \}\,.
\]
An efficient algorithm is available to compute these approximations (cf. \cite{HeStie52}). 
{\sc CGNE} with the discrepancy principle as a stopping rule is an order-optimal regularisation scheme for all $\mu >0$ 
(cf. Theorem~7.12 in \cite{EHaNeu96},\cite{Nemi86}).
And, due to its definition, {\sc CGNE} is the fastest to satisfy the discrepancy principle with respect to all regularisation schemes that compute approximations in the Krylov subspace $\mathcal{K}_m$. The analysis of CGNE with respect to its regularisation properties is involved, since the operators $R_m$ are nonlinear and not necessarily continuous (cf. Theorem~7.6 in \cite{EHaNeu96}, \cite{EiLouPla90}).

In this paper, we will define a method of the same type, but for the {\em shift-and-invert} or {\em resolvent} Krylov subspace
\begin{eqnarray} \label{ResKry} 
 \mathcal{Q}_m &= \mathcal{K}_m
    \left( 
      \left(
        I+T^*T\slash \gamma
      \right)^{-1}, T^*y^\delta
    \right) \\
     &=
   \mathrm{span} 
   \left\{
       T^*y^\delta, \left( I + T^*T\slash \gamma \right)^{-1}T^*y^\delta, \cdots, 
       \left( I + T^*T\slash \gamma \right)^{-m+1} T^*y^\delta 
    \right\}\,, \nonumber 
\end{eqnarray}
where $\gamma > 0$ is a fixed real number (e.g. \cite{BotKni20,Gri12,GGautosmooth17,LiuChuan19,moretnovati04,MoNo19,RamRei19,Ruhe84}). We define our method by 
\begin{equation} \label{minprop}
  x_m^\delta =: R_my^\delta, \qquad x_m^\delta = \mbox{argmin}_{x \in \mathcal{Q}_m} \|y^\delta - Tx\|,
  \qquad m=1,2,\ldots \,,
\end{equation}
combined with the discrepancy principle as its stopping rule. (The minimizer $x_m^\delta$ is uniquely defined, cf. lemma~\ref{alg:minprop} below.)
The subspace $\mathcal{Q}_m$ belongs to the class of rational Krylov subspaces which have been studied in recent years (e.g. references in \cite{TanjaDiss,guettelpole13}).
Since this method can be seen as solving  the normal equation $(\ref{normalequation})$ approximatively in the shift-and-invert Krylov subspace $\mathcal{Q}_m$, the method will be called {\em Shift-and-Invert on the Normal Equation} {\sc (SINE)}.
 Several regularisation schemes have been proposed that compute approximations in the subspace $\mathcal{Q}_m$, cf. example~\ref{methodswithqm}. 
By definition, our method will be the fastest to stop with respect to the discrepancy principle. 
SINE is related (but not identical) to {\sc CGNE} preconditioned by $(I+T^*T\slash \gamma)^{-1}$. Actually, SINE is not a preconditioning technique in the usual sense. Nevertheless, 
rational Krylov subspaces have been observed of being capable of accelerating the convergence (e.g.~\cite{ggaccel17}).  
The analysis of SINE with respect to its regularisation properties shares the difficulties of the analysis of CGNE, the family of operators $R_m$ is again nonlinear and not continuous in general, which can be seen by generalising the ideas of the proof of theorem~7.6 in \cite{EHaNeu96} and \cite{EiLouPla90}.

\begin{example} \label{methodswithqm}
Some regularisation schemes with approximations in the subspace $\mathcal{Q}_m$.
\begin{enumerate}
  \item Iterated Tikhonov-Phillips regularisation (cf. \cite{BuDoRei17,Fakeev81,KingChill79,Krjanev74}).
  \item Applying the implicit Euler method, the implicit midpoint-rule, or the trapezoidal rule with fixed time-step to asymptotic regularisation (Showalter's regularisation) leads to approximations in $\mathcal{Q}_m$ (cf. \cite{Rieder05}).
  \item The method proposed by Riley in \cite{Riley55} applied to the normal equation.
  \item The rational Arnoldi approach proposed in \cite{BreNoRe12}.
%  \item Semi-iterative methods???
\end{enumerate}
\end{example}

In section~\ref{basicprop}, we will show that $x_m^\delta$ in $(\ref{minprop})$ can be computed efficiently and discuss some basic properties of the method. Convergence for unperturbed data is shown in section~\ref{sec:conv} before the SINE method is discussed with respect to its regularisation properties in section~\ref{sec:sine}.
In section~\ref{sec:stop}, upper bounds on the number of iterations of SINE are discussed by comparing them to the known upper bounds on the number of iterations of CGNE.
 The findings are illustrated by an experiment in section~\ref{sec:numexp}.

Throughout we will use notations identical or closely related to the notations in 
\cite{EHaNeu96}. In particular, the functional calculus described in section~2.3 of \cite{EHaNeu96} is used without further note. Our proofs will follow closely or sometimes literally the corresponding proofs for {\sc CGNE} in chapter~7 of \cite{EHaNeu96}.

\section{Basic properties} \label{basicprop}

We consider algorithm~\ref{alg:sine}, where we choose $x_0^\delta=0$ without loss of generality. If $x_0^\delta$ were not zero, the corresponding shift-and-invert Krylov subspace would be spanned with $r_0=y^\delta-Tx_0^\delta$ instead of $y^\delta$, which allows to use prior information on the solution.  
First, we aim for the following properties:
Algorithm~\ref{alg:sine} computes $x_m^\delta$ according to $(\ref{minprop})$, as long as the algorithm does not break down. If Algorithm~\ref{alg:sine} breaks down in step $\kappa$ with $q_\kappa=0$, we have $x_{\kappa}^\delta=T^+y^\delta$ as well as $x_m^\delta=T^+y^\delta$ for $m \geq \kappa$ in $(\ref{minprop})$. 

\begin{algorithm}
\begin{algorithmic}
\STATE Choose $x_0^\delta$, set $r_0=y^\delta-Tx_0^\delta$, $w_0=T^*r_0$.
\FOR{$j=0,1,2,\ldots $}
\STATE $q_j = Tw_j$
\STATE $\delta_j=(q_j,q_j)$
\STATE $\alpha_j = \left( r_j, q_j \right) \slash \delta_j $
\STATE $x_{j+1}^\delta = x_j^\delta+\alpha_j w_j$
\STATE $r_{j+1} = r_j-\alpha_j q_j$
\STATE $s_j=T^*q_j$
\STATE $t_{j+1} = (I+T^*T\slash \gamma)^{-1} T^* r_{j+1}$
\STATE $\beta_j=
           \left( t_{j+1}, s_j \right) 
        \slash 
           \delta_j
        $
\STATE  $w_{j+1}=t_{j+1}-\beta_j w_j$
\ENDFOR
\end{algorithmic}
\caption{Shift-and-Invert on the Normal Equation {\sc (SINE)}}
\label{alg:sine}
\end{algorithm}

\begin{samepage}
\begin{lemma} \label{alg:rec}
As long as $q_{m-1} \not=0$
\begin{itemize}
\item[(i)] $(r_m,q_j)=(T^*r_m,w_j)=0$, $j=0,\ldots,m-1$
\item[(ii)] $\alpha_j \neq 0$, $j=0,\ldots,m-1$	
\item[(iii)] $(q_m,q_j)=0$, $j=0,\ldots,m-1$ 
\end{itemize}
\end{lemma}
\end{samepage}
\begin{proof}
The proof is via induction on $m$. For $m=1$, we have
\[
 (r_1,q_0)=(r_0,q_0)-\alpha_0(q_0,q_0)=0, \qquad 
 \alpha_0=\frac{(r_0,q_0)}{(q_0,q_0)}\,.
\]
Assume $\alpha_0=0$ then $(r_0,q_0)=0$ (since $(q_0,q_0) \neq 0$). It follows
\[
 0=(r_0,q_0)=(y^\delta,TT^*r_0)=(T^* r_0,T^* r_0)	
\]
and therefore $w_0=T^*r_0=0$ and $q_0=Tw_0=0$. This is a contradiction to our assumption $q_0 \neq 0$ and therefore $\alpha_0 \neq 0$ needs to be correct. Further,
\begin{eqnarray*}
(q_1,q_0) &=& (Tw_1,Tw_0) = (Tt_1-\beta_0Tw_0,Tw_0) \\
          &=& (Tt_1,Tw_0)-\beta_0(Tw_0,Tw_0), \qquad 
            \beta_0=\frac{(t_1,s_0)}{(q_0,q_0)} \\
          &=& (Tt_1,q_0)-(t_1,T^*q_0)=0\,,
\end{eqnarray*}
which concludes the proof of our statements for $m=1$.
Now we assume that the assertions are satisfied for $m$. Then, we have
\begin{eqnarray*}
(r_{m+1},q_j) &=& (r_m,q_j)-\alpha_m(q_m,q_j), \qquad
                \alpha_m=\frac{(r_m,q_m)}{(q_m,q_m)} \\
              &=& 
   \left\{
     \begin{array}{lcc}
       (r_m,q_m)-\alpha_m(q_m,q_m)=0, &~~~~& j=m,\\
        0, &~~~~& j<m
     \end{array}
   \right.\,.
\end{eqnarray*}
$\alpha_j \neq 0$ for $j=0,\ldots,m-1$ follows from the induction hypothesis. We show again by contradiction, that $\alpha_m$ can not be zero. Assume $\alpha_m=0$, hence $(r_m,q_m)=0$. We have
\begin{eqnarray*}
 0 &=& (T^*r_m,w_m)=(T^*r_m,t_m-\beta_{m-1}w_{m-1})=(T^*r_m,t_m) \\
   &=& ((I+T^*T\slash \gamma) t_m,t_m)=(t_m,t_m)+\frac{1}{\gamma}(Tt_m,Tt_m) \\
   &=& \|t_m\|^2+\frac{1}{\gamma}\|Tt_m\|^2\,.	
\end{eqnarray*}	
	From this, we immediately have $t_m=0$, $\beta_{m-1}=0$, $w_m=0$ and finally the contradiction $q_m=Tw_m=0$. Since this is not true, due to our assumption, $\alpha_m\neq 0$ is proved. Now, we target $(iii)$. It follows
\begin{eqnarray*}
(q_{m+1},q_m) &=& (Tt_{m+1}-\beta_mTw_m, q_m) \\
 &=& (Tt_{m+1},q_m)-\beta_m(q_m,q_m), \qquad 
    \beta_m=\frac{(t_{m+1},T^*q_m)}{(q_m,q_m)} \\
 &=& 0
\end{eqnarray*}
and for $j < m$, we have
\begin{eqnarray}
 (q_{m+1},q_j) &=& (Tt_{m+1},Tw_j)-\beta_m (q_m,q_j) \nonumber \\
  &=& (Tt_{m+1},Tw_j)=(T(I+T^*T \slash \gamma)^{-1}T^*r_{m+1},Tw_j) \label{lemmabasicprop1}\\
  &=& (T^* r_{m+1}, (I+T^*T \slash \gamma)^{-1}T^*q_j)\,. \nonumber
\end{eqnarray}
For $1 \leq j < m$, it follows from the iteration and $\alpha_j \neq 0$ that
\begin{eqnarray*}
(I+T^*T \slash \gamma)^{-1}T^*q_j
	&=& -\frac{1}{\alpha_j}\left[ t_{j+1}-t_j\right] \\	  
	&=& -\frac{1}{\alpha_j}\left[ w_{j+1}+(\beta_j-1)w_j-\beta_{j-1}w_{j-1} \right] \\
	&=& -\frac{1}{\alpha_j}w_{j+1}-\frac{\beta_j-1}{\alpha_j}w_j+\frac{\beta_{j-1}}{\alpha_j}w_{j-1}\,.
\end{eqnarray*}
Inserting this into $(\ref{lemmabasicprop1})$ gives
\begin{eqnarray*}
(q_{m+1},q_j) &=& -\frac{1}{\alpha_j}(T^*r_{m+1},w_{j+1})-\frac{\beta_j-1}{\alpha_j}(T^*r_{m+1},w_j) \\
              && \hspace{5.9cm} +\frac{\beta_{j-1}}{\alpha_j}(T^*r_{m+1},w_{j-1}) \\
	      &=& -\frac{1}{\alpha_j}(r_{m+1},q_{j+1})-\frac{\beta_j-1}{\alpha_j}(r_{m+1},q_j)+
	       \frac{\beta_{j-1}}{\alpha_j}(r_{m+1},q_{j-1}) \\
	      &=& 0\,.
\end{eqnarray*}
For $j=0$, a simple but tedious calculation shows that
\[
 t_0:=(I+T^*T \slash \gamma)^{-1}T^*q_0=c_1 w_1 + c_2 w_0, \qquad c_1,c_2 \in \mathbb{R}\,.	
\]
With this, $(q_{m+1},q_0)=0$ follows as above with different coefficients.
\end{proof}
\begin{lemma} \label{alg:basisprop}
 As long as algorithm~\ref{alg:sine} does not break down with $q_m=0$, we have
\[
 \mathcal{Q}_{m+1}=\mbox{span} \left\{ w_0, \ldots, w_m \right\}\,,
\]
where $w_0,\ldots, w_m$ is a basis of $\mathcal{Q}_{m+1}$.
\end{lemma}
\begin{proof}
It follows readily from algorithm~\ref{alg:sine} that
\[
	M:=\mbox{span} \left\{ w_0, \ldots, w_m \right\} \subseteq \mathcal{Q}_{m+1}\,.
\]
Since $q_0=Tw_0,\ldots,q_m=Tw_m$ is an orthogonal basis of $TM$, due to $(iii)$ of lemma~\ref{alg:rec}, $w_0,\ldots,w_m$ are linearly independent. Since the dimension of $\mathcal{Q}_{m+1}$ is smaller or equal to $m+1$, $w_0,\ldots, w_m$ needs to be a basis of $\mathcal{Q}_{m+1}$.
\end{proof}
\begin{lemma} \label{alg:minprop}
The iterates $x_m^\delta$ of algorithm~\ref{alg:sine} satisfy $(\ref{minprop})$. 
\end{lemma}
\begin{proof}
Due to algorithm~\ref{alg:sine} with $x_0^\delta=0$, we have $x_m^\delta \in \mbox{span}\{ w_0,\ldots,w_{m-1}\}=\mathcal{Q}_m$. Now, 
let $z_m \in \mathcal{Q}_m=\mbox{span}\{ w_0,\ldots,w_{m-1}\}$ such that  $z_m \neq x_m^\delta \in \mathcal{Q}_m$. Hence,  we can write
\[
0 \neq z_m-x_m^\delta = \sum_{j=0}^{m-1} \xi_j w_j, \qquad \xi_j \in \mathbb{R},
\]
and obtain
\begin{eqnarray*}
  \|y^\delta-Tz_m\|^2 &=& \|y^\delta-Tx_m^\delta\|^2-2\sum_{j=0}^{m-1} \xi_j (Tw_j,r_m)+
  \|T\sum_{j=0}^{m-1} \xi_j w_j\|^2 \\
  & > & \|y^\delta-Tx_m^\delta\|^2-2\sum_{j=0}^{m-1} \xi_j (q_j,r_m)=\|y^\delta-Tx_m^\delta\|^2 
\end{eqnarray*}
by lemma~\ref{alg:rec}. The strict inequality, and therefore uniqueness of the minimizer, follows since $\mathcal{Q}_m \subseteq \mathcal{N}(T)^\perp$ and  $\|T\sum_{j=0}^{m-1} \xi_j w_j\|^2 > 0$ as a consequence.
\end{proof}
\begin{lemma}\label{alg:breakdown} If algorithm~\ref{alg:sine} breaks down in step $\kappa$ with $q_\kappa=0$, then $x_\kappa^\delta=x^+=T^+y^\delta$.
\end{lemma}
\begin{proof}
We first show that $q_\kappa=0$ means $T^*r_\kappa=0$. First assume $\kappa=0$.  Then $0=(r_0,q_0)=(r_0,TT^*r_0)=\|T^*r_0\|^2$, hence $T^*r_0=0$. Now let $\kappa>0$ and assume $(r_\kappa,q_\kappa)=0$. Then, with the help of statement $(i)$ of lemma~\ref{alg:rec}
\begin{eqnarray*}
0 &=& (T^*r_\kappa,w_\kappa)=(T^*r_\kappa,t_\kappa-\beta_{\kappa-1}w_{\kappa-1})=(T^*r_\kappa,t_\kappa) \\
  &=& ((I+T^*T\slash \gamma)t_\kappa,t_\kappa)=(t_\kappa,t_\kappa)+\frac{1}{\gamma}(Tt_\kappa,Tt_\kappa)\\
  &=& \|t_\kappa \|^2 + \frac{1}{\gamma}\|Tt_\kappa\|^2\,.
\end{eqnarray*}
Hence we have $0=(I+T^*T\slash \gamma)t_\kappa=T^*r_\kappa$ in all cases. Since $r_\kappa=y^\delta-Tx_k^\delta$, this means
\[
  T^*Tx_\kappa^\delta=T^*y^\delta, \qquad x_\kappa^\delta \in \mathcal{Q}_\kappa \subseteq \mathcal{N}(T)^\perp,
\]
which characterises the minimum-norm solution, that is $x_\kappa^\delta=x^+$ (cf. Theorem~2.5 and Theorem~2.6 in \cite{EHaNeu96}).
\end{proof}
If the algorithm stops with $q_\kappa=0$, it follows from $(\ref{minprop})$ that $x_m^\delta=x_\kappa^\delta$, $m \geq \kappa$.

Analogous to the description of the Krylov subspace $\mathcal{K}_m$ with the set $\Pi_{m-1}$ of polynomials of degree less than $m$, the shift-and-invert Krylov subspace $\mathcal{Q}_m$ can be described with the help of rational functions as
\[
\mathcal{Q}_m = \left\{ r(T^*T)T^*y^\delta~|~r \in \Pi_{m-1}\slash (1+\cdot \slash \gamma)^{m-1} \right\}\,.
\]
The functional calculus of section~2.3 in \cite{EHaNeu96} applies to these rational functions. The iterates, residui etc. of algorithm~\ref{alg:sine} can be identified with the corresponding rational functions (cf. \cite{EHaNeu96,Hanke95}). The following lemma describes some properties of the rational function $r_m(\lambda)$ that belongs to the residuum $r_m$, i.e.,
the function $r_m(\lambda)$ such that
\begin{equation} \label{resrep}
r_m=y^\delta-Tx_m=r_m(TT^*)y^\delta \qquad \mbox{or} \qquad T^*r_m=r_m(T^*T)T^*y^\delta\,,
\end{equation}
respectively. 
\begin{lemma} \label{interlacing} 
As long as the stopping index $\kappa$ has not been reached, we have  
\begin{equation} \label{rdetail}
 r_m(\lambda)=\frac{p_m(\lambda)}{(1+\lambda \slash \gamma)^{m-1}} 
 \quad \mbox{with} \quad
 p_m(\lambda)=\prod_{j=1}^m 
 \left(
  1-\frac{\lambda}{\lambda_{j,m}}
 \right), \quad m \geq 1\,,
\end{equation}
where $r_m(\lambda)$ is the rational function that describes the residuum $r_m$ in algorithm~\ref{alg:sine}.
The values $\lambda_{j,m-1}$, $j=1,\ldots,m-1$ of $r_{m-1}(\lambda)$ and the values 
$\lambda_{j,m}$, $j=1,\ldots,m$ of $r_m(\lambda)$ are interlacing, real, and positive, that is
\[
 0 < \lambda_{1,m} < \lambda_{1,m-1} < \lambda_{2,m} < \cdots<  \lambda_{m-1,m} < \lambda_{m,m-1} < \lambda_{m,m} \leq \|T\|^2\,.
\]
\end{lemma}
\begin{proof}
Let $v_0, \cdots, v_{m-1}$ be an orthonormal basis such that
\[
  \mathcal{Q}_\ell = 
  \mbox{span} \{ w_0, \cdots, w_{\ell-1} \}
  = 
  \mbox{span} \{ v_0, \cdots, v_{\ell-1} \}
  \qquad
  \mbox{for} 
  \quad \ell=1,\ldots,m\,.
\]
The basis $v_j$ might be obtained from the basis $w_0,\cdots,w_{m-1}$ by the Gram--Schmidt process (or the Arnoldi process with $T^*y^\delta$). 
Then we can represent the iterate $x_\ell^\delta$ as 
$
   x_\ell^\delta=\sum_{j=0}^{\ell-1} z_{j,\ell} v_j
$.
Since $T^*y^\delta-T^*Tx_\ell^\delta \perp \mathcal{Q}_\ell$ is an equivalent condition to $x_\ell^\delta$ being the minimizer in $(\ref{minprop})$, we have
\[
 S_\ell=\left( (T^*Tv_i,v_j) \right)_{j,i=0}^{\ell-1}, \qquad S_\ell z_\ell = \beta e_1, \qquad \beta=\|T^*y^\delta\|\,.
\]
Since $\mathcal{Q}_\ell \subset \mathcal{N}(T)^\perp=\mathcal{N}(T^*T)^\perp$, $S_\ell$ is invertible, and therefore symmetric positive definite with $\|S_\ell\| \leq \|T\|^2$ and $z_\ell=\beta S_\ell^{-1}e_1$. 
Specifically,
\begin{equation} \label{rlemma:eq1}
x_m^\delta=\sum_{j=0}^{m-1} z_{j,m} v_j, \qquad \mbox{with} \qquad z_m=\beta S_m^{-1} e_1\,.
\end{equation}
Furthermore,
the $(\ell-1,\ell-1)$ submatrix of $S_\ell$ is $S_{\ell-1}$ for $\ell=2,\ldots,m$.
Inductively, by the interlacing eigenvalue theorem (cf. Theorem~3.6 in \cite{Stewartbook01}), this leads to the result 
that the eigenvalues of $S_\ell$ are separated and interlaced with the eigenvalues of $S_{\ell-1}$. If we designate the eigenvalues of $S_\ell$ with $\lambda_{1,\ell} < \cdots < \lambda_{\ell,\ell}$ then we obtain the statement on the interlacing of the numbers in our theorem.
We still have to show that the $\lambda_{j,m}$, $j=1,\ldots,m$ are the zeros of our function $r_m(\lambda)$. 
Using the notation of quasi-matrices in the standard representation of the iterate in the rational Krylov subspace (cf. \cite{GG13,stewart}), one obtains
\begin{equation} \label{rlemma:eq2}
 x_m^\delta=\frac{q_{m-1}(T^*T)}{(1+T^*T\slash \gamma)^{m-1}}T^*y^\delta =\sum_{j=0}^{m-1}\xi_{j,m} v_j,
 \quad
\xi_m=\beta \frac{q_{m-1}(S_m)}{(1+S_m \slash \gamma)^{m-1}} e_1\,.
\end{equation}
By comparing $(\ref{rlemma:eq1})$ with $(\ref{rlemma:eq2})$,  we obtain
\[
  \beta \frac{q_{m-1}(S_m)}{(1+S_m \slash \gamma)^{m-1}} e_1
  =
  \beta S_m^{-1}e_1 
  \qquad
  \mbox{and hence}
  \qquad
  \frac{q_{m-1}(S_m)}{(1+S_m \slash \gamma)^{m-1}}
  =
  S_m^{-1}\,,
\]
since $S_m$ comes from a Krylov process and the minimal polynomial of $S_m$ with respect to $e_1$ has therefore degree $m$. Finally, we obtain
\[
  \frac{q_{m-1}(\lambda_{j,m})}{(1+\lambda_{j,m} \slash \gamma)^{m-1}}
  =\frac{1}{\lambda_{j,m}} \qquad \mbox{for} \qquad j=1,\ldots,m
\]
by spectral decomposition.
This means that $r_m(\lambda)=1-\lambda q_{m-1}(\lambda) \slash (1+ \lambda \slash \gamma)^{m-1}$
has zeros $\lambda_{j,m}$, $j=1,\ldots,m$. Together with the obvious value $r_m(0)=1$, this shows the representation of $r_m(\lambda)$ as given in our lemma.
\end{proof}

The inner product introduced in the following lemma will be crucial for the proof of our main theorem. Also, the idea of algorithm~\ref{alg:sine} can be briefly stated as computing an orthogonal basis $w_0,\ldots w_{j-1}$ of the rational Krylov subspace $\mathcal{Q}_j$ with respect to the inner product $[\cdot,\cdot]$ in $(\ref{rperpQm})$ when the vectors are identified with the corresponding rational functions.
In the following, we will often not make a difference between the residuum $r_m$ and the rational function $r_m(\lambda)$ and denote both by $r_m$.
\begin{lemma} \label{rperpq} 
The rational functions $r_m$ generated by Algorithm~\ref{alg:sine}
are orthogonal to $\Pi_{m-1} \slash (1+ \cdot \slash \gamma)^{m-1}$ with respect to the inner product
\begin{equation} \label{rperpQm}
  \big[ \varphi, \psi \big] = \int_0^{\|T\|^2+} 
      \varphi(\lambda) \psi(\lambda) \lambda \, d\|F_\lambda y^\delta\|^2\,, 
\end{equation}
where $F_\lambda$ designates the spectral family of $TT^*$.
Among all rational $\varphi \in \Pi_m \slash (1+ \cdot \slash \gamma)^{m-1}$ with $\varphi(0)=1$, $r_m$ minimises the functional
\begin{equation} \label{minfunkt}
\Phi[\varphi]= \int_0^{\|T\|^2+} 
      \varphi^2(\lambda) \, d\|F_\lambda y^\delta\|^2\,.
\end{equation}
\end{lemma}
\begin{proof}
We have
\[
[\varphi, \psi]=\int_0^{\|T\|^2+} \varphi(\lambda)\psi(\lambda)\lambda \,d\|F_\lambda y^\delta\|^2 
     =(\varphi(T^*T)T^*y^\delta,\psi(T^*T)T^*y^\delta)\,
\]
which gives the first assertion by lemma~\ref{alg:rec} $(i)$. The second assertion follows by lemma~\ref{alg:minprop}.
\end{proof}

\section{Convergence} \label{sec:conv}

The following theorem shows convergence of the iterates $x_m$ in algorithm~\ref{alg:sine} to the minimum-norm solution $x^+=T^+y$ for data $y \in \mathcal{D}(T^+)$ and our general choice $x_0=0$. For an initial guess $x_0 \ne 0$, 
it can be readily shown that the iterates converge to $T^+y + P_{\mathcal{N}(T)}x_0$, where $P_{\mathcal{N}(T)}$ is the orthogonal projector to the null space of $T$. The superscript $\delta$ has been dropped in this section in order to 
emphasise that data $y\in \mathcal{D}(T^+)$ without perturbation is considered.  
\begin{theorem} \label{sineconv} 
The sequence of SINE iterates $\{ x_n \}$ converge to $T^+y$ for all $y \in \mathcal{D}(T^+)$. 
\end{theorem}
\begin{proof}
We basically follow the lines of the proof of theorem~7.9 in \cite{EHaNeu96} or \cite{NemPoI}, respectively.
If the iteration terminates after a finite number of steps then the corresponding iterate coincides with $T^+y$ according to lemma~\ref{alg:breakdown}. 
We therefore assume, that the iteration does not terminate. Then we have the ordering
\[
0 < \lambda_{1,m} < \lambda_{2,m} < \cdots < \lambda_{m,m} \leq \|T\|^2 
\]
of the Ritz values according to lemma~\ref{interlacing}. 
From the representation of the residual rational function $(\ref{rdetail})$ in lemma~\ref{interlacing},
we obtain
\begin{equation} \label{rd0}
 |r'_m(0)|=-r_m'(0)=\sum_{j=1}^m \frac{1}{\lambda_{j,m}}+\frac{m-1}{\gamma}
\end{equation}
by the same calculations that lead to $(\ref{rmdiffhelp})$ below.
Since $r_m\slash (\lambda-\lambda_{1,m})$ is in the space $\Pi_{m-1}\slash(1+\cdot \slash \gamma)^{m-1}$, the 
orthogonality relation $(\ref{rperpQm})$ yields
\[
0=\int_0^{\|T\|^2+} r_m(\lambda)\frac{r_m(\lambda)}{\lambda-\lambda_{1,m}} \lambda\,d\|F_\lambda y\|^2\, 
\]
which gives
\[
\int_0^{\lambda_{1,m}} r_m^2(\lambda)\frac{\lambda}{\lambda_{1,m}-\lambda}\,d\|F_\lambda y\|^2
=
\int_{\lambda_{1,m}}^{\|T\|^2+} r_m^2(\lambda)\frac{\lambda}{\lambda-\lambda_{1,m}}\,d\|F_\lambda y\|^2\,.
\]
Since $\lambda \slash (\lambda-\lambda_{1,m}) \geq 1$ for $\lambda \geq \lambda_{1,m}$ we obtain 
\[
\int_0^{\lambda_{1,m}} r_m^2(\lambda)\frac{\lambda}{\lambda_{1,m}-\lambda}\,d\|F_\lambda y\|^2 
\geq
\int_{\lambda_{1,m}}^{\|T\|^2+} r_m^2(\lambda)\,d\|F_\lambda y\|^2\,.
\]
And therefore,
\begin{eqnarray*}
\|y-Tx_m\|^2 
    &= 
    \int_0^{\lambda_{1,m}} r_m^2(\lambda)\,d\|F_\lambda y\|^2 
    +
    \int_{\lambda_{1,m}}^{\|T\|^2+} r_m^2(\lambda)\,d\|F_\lambda y\|^2 \\
    &\leq 
    \int_0^{\lambda_{1,m}} r_m^2(\lambda) 
     \Big( 1+\frac{\lambda}{\lambda_{1,m}-\lambda} \Big) \, d\|F_\lambda y\|^2\,.
\end{eqnarray*}
Defining
\[
\varphi_m(\lambda) := r_m(\lambda)
      \Big(
       \frac{\lambda_{1,m}}{\lambda_{1,m}-\lambda}
      \Big)^{\frac{1}{2}}, \qquad 0 \leq \lambda \leq \lambda_{1,m}\,,
\]
we obtain the estimate
\begin{equation} \label{resbound}
  \|y-Tx_m\| \leq \|F_{\lambda_{1,m}}\varphi_m(TT^*)y\| 
  \leq \max_{0 \leq \lambda \leq \lambda_{1,m}}
   \sqrt{\lambda \varphi_m^2(\lambda)} \quad \|E_{\lambda_{1,m}} x^+ \| \,,
\end{equation}
where $E_\lambda$ designates the spectral family of $T^*T$. For the last inequality, we additionally assumed 
$y \in \mathcal{R}(T)$, that is, $y=Tx^+$.
It follows immediately, that $ 0 \leq \varphi_m(\lambda) \leq 1$ for $\lambda \in [0, \lambda_{m,1}]$.
We further calculate
\begin{eqnarray*}
  \frac{d}{d\lambda} \varphi_m^2(\lambda) 
  &=
  2\varphi_m(\lambda)\varphi_m'(\lambda) \\
  &=
  2 \frac{p_m(\lambda)}{(1+\lambda \slash \gamma)^{m-1}}
  \left(
   \frac{\lambda_{m,1}}{\lambda_{m,1}-\lambda}
  \right)
  \left[
    \left(
        \frac{p_m(\lambda)}{(1+\lambda \slash \gamma)^{m-1}}
    \right)'
   \right. \\
  & \qquad \qquad \qquad \qquad  +
  \left.
   \frac{1}{2} \cdot
    \left(
        \frac{p_m(\lambda)}{(1+\lambda \slash \gamma)^{m-1}}
    \right)
    \cdot
    \frac{1}{(\lambda_{m,1}-\lambda)}
  \right]
\end{eqnarray*}
By inserting in
\begin{eqnarray*}
  \left(
     \frac{p_m(\lambda)}{(1+\lambda \slash \gamma)^{m-1}}
  \right)'
  &=
  \frac{p_m'(\lambda)}{(1+\lambda \slash \gamma)^{m-1}}
  -
  p_m(\lambda) \cdot \frac{m-1}{\gamma} \cdot
  \frac{1}{(1+\lambda \slash \gamma)^m}
\end{eqnarray*}
the relation
\[
  p_m'(\lambda) = -p_m(\lambda) \cdot \sum_{j=1}^m \frac{1}{\lambda_{j,m}-\lambda}\,,
\]
we obtain
\begin{eqnarray} \label{rmdiffhelp}
  r_m'(\lambda) 
  &=
  -\frac{p_m(\lambda)}{(1+\lambda \slash \gamma)^{m-1}}
  \left[
  \sum_{j=1}^m \frac{1}{\lambda_{j,m}-\lambda} +
  \frac{m-1}{\gamma} \cdot \frac{1}{1+\lambda \slash \gamma}
  \right] 
\end{eqnarray}
and therefore
\begin{eqnarray*}
  \frac{d}{d\lambda} \varphi_m^2(\lambda) 
  &=
  2\varphi_m(\lambda)\varphi_m'(\lambda) \\
  &=
  \left(
  \frac{p_m(\lambda)}{(1+\lambda \slash \gamma)^{m-1}}
  \right)^2
  \left(
   \frac{\lambda_{m,1}}{\lambda_{m,1}-\lambda}
  \right)
  \left[
   -2\sum_{j=1}^m \frac{1}{\lambda_{j,m}-\lambda}
   \right. \\
   & \qquad \qquad \qquad \qquad \qquad \quad
   \left. 
   -2\cdot \frac{m-1}{\gamma} \cdot \frac{1}{1+\lambda \slash \gamma} 
   + \frac{1}{\lambda_{1,m}-\lambda}
  \right]\,.
\end{eqnarray*}
Altogether, with $\nu > 0$, 
\begin{eqnarray*}
  &\frac{d}{d\lambda} \lambda^\nu \varphi^2_m(\lambda) \\
  &=
  \nu \lambda^{\nu-1} \varphi^2_m(\lambda)+\lambda^\nu \varphi^2_m(\lambda)
  \left(
    \frac{1}{\lambda_{1,m}-\lambda}-\sum_{j=1}^m \frac{2}{\lambda_{j,m}-\lambda}     -2 \cdot \frac{m-1}{\gamma} \cdot \frac{1}{1+\lambda \slash \gamma}
  \right) \\
  &=
  \lambda^{\nu-1}\varphi^2_m(\lambda) \cdot
  \left[ 
    \nu + \lambda
  \left(
    \frac{1}{\lambda_{1,m}-\lambda}-\sum_{j=1}^m \frac{2}{\lambda_{j,m}-\lambda}     -2 \cdot \frac{m-1}{\gamma} \cdot \frac{1}{1+\lambda \slash \gamma}
  \right) 
  \right].
\end{eqnarray*}
Since $0^\nu\varphi^2_m(0)=0=\lambda_{1,m}^\nu\varphi^2_m(\lambda_{m,1})$, there is at least one $0<\lambda^*<\lambda_{m,1}$, such that
$(\lambda^\nu \varphi^2_m(\lambda))'(\lambda^*)=0$ and such that the maximum is achieved at this point. Hence the equation
\begin{eqnarray}
 \nu &= 
   \lambda^* 
   \left(
    \sum_{j=1}^m \frac{2}{\lambda_{j,m}-\lambda^*} - \frac{1}{\lambda_{1,m}-\lambda^*} + 2\cdot \frac{m-1}{\gamma}\cdot \frac{1}{1+\lambda^* \slash \gamma}
   \right) \label{maxstellenconf} 
\end{eqnarray}
holds true. 
We need to distinguish two cases. First case: $\gamma \geq \|T\|^2$. Then, we have $0 < \lambda^* < \lambda_{1,m} \leq \|T\|^2 \leq \gamma$. Since $\lambda^* < \gamma$, we have
\[
  2\cdot \frac{m-1}{\gamma}\cdot \frac{1}{1+\lambda^*\slash \gamma} 
  \geq 
  \frac{m-1}{\gamma}
\]
and hence
\begin{eqnarray*}
 \nu &= 
   \lambda^* 
   \left(
     \sum_{j=1}^m \frac{2}{\lambda_{j,m}-\lambda^*}
     -\frac{1}{\lambda_{1,m}-\lambda^*}
     +2 \cdot \frac{m-1}{\gamma}\cdot\frac{1}{1+\lambda^*\slash \gamma}
   \right) \\
   &\geq
   \lambda^* 
   \left(
     \sum_{j=1}^m \frac{1}{\lambda_{j,m}-\lambda^*}
     +\frac{m-1}{\gamma}
   \right) 
   \geq
   \lambda^* 
   \left(
     \sum_{j=1}^m \frac{1}{\lambda_{j,m}}
     +\frac{m-1}{\gamma}
   \right) 
   \\
   &=
   \lambda^* \left( -r_m'(0) \right)
\end{eqnarray*}
and therefore
\[
  \lambda^* \leq \frac{\nu}{-r_m'(0)}=\frac{\nu}{|r_m'(0)|}\,.
\]
Second case: $\gamma < \|T\|^2$. We set 
\[
 p=\frac{1}{2} \left( \frac{\|T\|^2}{\gamma} + 1 \right) > 1\,.
\]
We then have 
\[
  2\cdot \frac{m-1}{\gamma} \cdot \frac{1}{1+\lambda^*\slash \gamma} \geq 
  \frac{1}{p} \cdot \frac{m-1}{\gamma}\,.
\]
We can therefore conclude
\begin{eqnarray*}
 \nu &= 
   \lambda^* 
   \left(
     \sum_{j=1}^m \frac{2}{\lambda_{j,m}-\lambda^*} 
     -\frac{1}{\lambda_{1,m}-\lambda^*} 
     + 2\cdot\frac{m-1}{\gamma}\cdot \frac{1}{1+\lambda^*\slash \gamma} 
   \right) \\
  & \geq
   \lambda^*
   \left(
    \sum_{j=1}^m \frac{1}{\lambda_{j,m}-\lambda^*} 
   +\frac{1}{p} \cdot \frac{m-1}{\gamma} 
   \right) \\
  & \geq
   \lambda^*
   \cdot \frac{1}{p} \cdot
   \left(
    \sum_{j=1}^m \frac{1}{\lambda_{j,m}-\lambda^*} 
   +\frac{m-1}{\gamma} 
   \right) 
   \geq
   \lambda^*
   \cdot \frac{1}{p} \cdot
   \left(
    \sum_{j=1}^m \frac{1}{\lambda_{j,m}} 
   +\frac{m-1}{\gamma} 
   \right) 
\\
  &=
   \lambda^* \cdot \frac{1}{p} \cdot \left( -r_m'(0) \right) 
\end{eqnarray*}
hence we have
\[
  \lambda^* \leq \frac{\nu p}{-r_m'(0)}=\frac{\nu p}{|r_m'(0)|}
  \quad \mbox{with} \qquad
  p=\frac{1}{2} \left( \frac{\|T\|^2}{\gamma} +1\right)\,.
\]
In both cases, we have
\[
  \lambda^* \leq c \cdot \frac{\nu}{|r_m'(0)|}, \qquad c=\max \{ 1,p \}\,.
\]
Hence
\begin{equation} \label{basephiest}
  \sup_{0 \leq \lambda \leq \lambda_{1,m}} \lambda^\nu \varphi_m^2(\lambda) 
  \leq (\lambda^*)^\nu \varphi^2_m(\lambda^*) \leq (\lambda^*)^\nu 
  \leq c^\nu \nu^\nu |r_m'(0)|^{-\nu}\,, \quad \nu > 0\,.
\end{equation}
We now relax the assumption on $y$ to $y \in \mathcal{D}(T^+)=\mathcal{R}(T) \oplus \mathcal{R}(T)^\perp$.
Since $\mathcal{R}(T)^\perp=\mathcal{N}(T^*)$, algorithm~\ref{alg:sine} produces the same iterates for $y \in \mathcal{R}(T) \oplus \mathcal{R}(T)^\perp$ and $P_{\overline{R(T)}}y\in \mathcal{R}(T)$, respectively. Rewriting $P_{\overline{R(T)}}y=Tx^+$ with $x^+=T^+y$, we can apply (\ref{basephiest}) with $\nu=1$ and obtain\[
 \|P_{\overline{R(T)}}y-Tx_m \|^2 \leq \|F_{\lambda_{1,m}}\varphi_m(TT^*)Tx^+\|^2 \leq c|r_m'(0)|^{-1} \|E_{\lambda_{1,m}}x^+\|^2\,.
\]
We now fix $0 < \epsilon \leq \lambda_{1,m}$ and obtain the estimates
\begin{eqnarray*}
  \|x^+-x_m\| 
    &= 
    \|r_m(T^*T)x^+\| \leq \|E_\epsilon r_m(T^*T)x^+\| 
      + \|(I-E_\epsilon)r_m(T^*T)x^+\| \\
    &\leq
     \|E_\epsilon r_m(T^*T)x^+\| 
    + \epsilon^{-\frac{1}{2}} \|(I-F_\epsilon)r_m(TT^*)P_{\overline{R(T)}}y\| \\
    &\leq 
     \|E_\epsilon x^+\|
    +\epsilon^{-\frac{1}{2}} \|P_{\overline{R(T)}}y-Tx_m\|\\
    &\leq
      \|E_\epsilon x^+\|+
      \left(
         \frac{c|r_m(0)|^{-1}}{\epsilon}
      \right)^{\frac{1}{2}}
      \|E_{\lambda_{1,m}}x^+\|\,.
\end{eqnarray*}
We have to consider two cases. If $\lambda_{1,m} \rightarrow 0$ as $m \rightarrow \infty$, then one can choose $\epsilon=\epsilon_m=\lambda_{1,m}$. Since $|r_m'(0)| \geq \lambda_{1,m}^{-1}$, we obtain
\[
 \|x^+-x_m\| \leq (1+\sqrt{c})\|E_{\lambda_{1,m}}x^+\| \rightarrow 0 \qquad \mbox{as} \qquad m \rightarrow \infty\,.
\]
In the second case, if $\lambda_{1,m} \rightarrow \lambda_1 > 0$ as $m \rightarrow \infty$, then we choose $\epsilon_m=|r_m'(0)|^{-\frac{1}{2}}$. Since $|r_m'(0)| \geq m\|T\|^{-2}$, $\epsilon_m \rightarrow 0$ as $m \rightarrow \infty$ and hence, $\epsilon_m < \lambda_{1,m}$ for $m$ sufficiently large. In this case we obtain
\[
  \|x^+-x_m\| \leq \|E_{\epsilon_m}x^+\| 
  + \left( c\epsilon_m \right)^{\frac{1}{2}}\|x^+\| \rightarrow 0 
  \qquad \mbox{as} \qquad
  m \rightarrow \infty\,.
\] 
Consequently, in any case, we have $\|x^+-x_m\| \rightarrow 0$ as $m \rightarrow \infty$.
\end{proof}

\section{{\sc SINE} is an order-optimal regularisation method} \label{sec:sine}
We assume that 
\[
  y \in \mathcal{R}(T), \qquad \|y^\delta-y\| \leq \delta\,,
\]
where the noise level $\delta > 0$ is known.  
Algorithm~\ref{alg:sine} is stopped with $m=m(\delta,y^\delta)$ according to the discrepancy
principle $(\ref{discprin})$. For the stopping index $m=m(\delta,y^\delta)\geq 1$,
\begin{equation} \label{discrep}
 \|y^\delta-Tx_{m(\delta,y^\delta)}^\delta\| \leq \tau \delta < \|y^\delta-Tx_{m(\delta,y^\delta)-1}^\delta\|
\end{equation}
is satisfied (with an a-priori chosen $\tau > 1$), for $m=0$, only the first inequality holds true. The algorithm always terminates after a finite number of steps, which can be seen as follows. Lemma~\ref{ressetlem} also holds for $\mu=0$ and $\rho=\|x^+\|$. Hence
\[
  \lim_{m\rightarrow \infty} \|y^\delta-Tx_m\| \leq \delta + \lim_{m\rightarrow \infty} c|r_m'(0)|^{-\frac{1}{2}}\|x^+\|=\delta\,,
\]
since $|r_m'(0)|^{-\frac{1}{2}} \rightarrow 0$, what we already know. 
The limit of the norm of the residuals exists, since the sequence is non-increasing due to lemma~\ref{alg:minprop} or $(\ref{minprop})$, respectively,  and bounded from below by zero. 
Since $\tau \delta > \delta$, the discrepancy principle stops the algorithm after a finite number of steps. 
If the algorithm has a finite termination index $\kappa$, then $q_\kappa=0$. According to lemma~\ref{alg:breakdown} we have $x_\kappa^\delta=T^+y^\delta$, in which case
\[
  \|y^\delta-Tx_\kappa^\delta\|=\|(I-P_{\overline{R(T)}})y^\delta\|
   =\|(I-P_{\overline{R(T)}})(y^\delta-y)\| \leq \delta
\]
and therefore $m(\delta,y^\delta) \leq \kappa$.

The letter $c$ designates a generic constant in the following lemmata and proofs.
\begin{lemma} \label{ressetlem}
Let $y=Tx^+$ with $x^+ \in \mathcal{X}_{\mu,\rho}$. Then for $0<m \leq \kappa$,
\[
  \|y^\delta - Tx_m^\delta \| \leq \delta +c|r_m'(0)|^{-\mu-\frac{1}{2}}\rho\,.
\]
\end{lemma}

\begin{proof}
The bound~$(\ref{resbound})$ proved in theorem~\ref{sineconv} reads
\[
  \|y^\delta-Tx_m^\delta\| \leq \|F_{\lambda_{1,m}}\varphi_m(TT^*)y^\delta\|\,.
\]
As before $\varphi_m$ is bounded by $1$ in $[0,\lambda_{1,m}]$ and satisfies
the equation~$(\ref{basephiest})$ with $\nu=2\mu+1$ 
\[
 \lambda^{2\mu+1} \varphi_m^2(\lambda) \leq c^{2\mu+1}(2\mu+1)^{2\mu+1}|r'_m(0)|^{-2\mu-1},\qquad 0\leq\lambda\leq \lambda_{1,m}\,.
\]
If we insert these estimates and use $y=T(T^*T)^\mu w$ with $\|w\| \leq \rho$, we obtain
\begin{eqnarray*}
  \|y^\delta-T x_m^\delta\| 
  &\leq
  \|F_{\lambda_{1,m}}\varphi_m(TT^*)(y^\delta-y)\|
     +\|F_{\lambda_{1,m}}\varphi_m(TT^*)y\| \\
  &\leq \delta + \|E_{\lambda_{1,m}}\varphi_m(T^*T)(T^*T)^{\mu+\frac{1}{2}}w\| \\
  &\leq \delta + c^{\mu+\frac{1}{2}}(2\mu+1)^{\mu+\frac{1}{2}}|r'_m(0)|^{-\mu-\frac{1}{2}}\rho\,,
\end{eqnarray*}
which gives the assertion.
\end{proof}
The following lemma and its proof are a nearly literal copy of Lemma~7.11 in \cite{EHaNeu96}, only that the functions representing the iteration are rational instead of polynomial. 

\begin{lemma} \label{errlem}
Assume that $y=Tx^+$ with $x^+ \in \mathcal{X}_{\mu,\rho}$. Then 
for $0 \leq m \leq \kappa$,
\[
  \|x_m^\delta-x^+\| \leq c(\rho^{\frac{1}{2\mu+1}}\delta_m^{\frac{2\mu}{2\mu+1}}+\sqrt{|r_m'(0)|}\delta_m)\,,
\]
where
\[
  \delta_m := \max\{ \|y^\delta-Tx_m^\delta \|,\delta\}\,.
\]
\end{lemma}
\begin{proof}
By the interpolation inequality (cf. \cite{EHaNeu96}, page 47) and $x_0^\delta=0$,
\[
 \|x^+\| \leq \rho^{\frac{1}{2\mu+1}} \|y\|^{\frac{2\mu}{2\mu+1}}
      \leq \rho^{\frac{1}{2\mu+1}}(\|y^\delta\|+\|y-y^\delta\|)^\frac{2\mu}{2\mu+1}
      \leq c \rho^{\frac{1}{2\mu+1}} \delta_0^\frac{2\mu}{2\mu+1}\,.
\]
We conclude that the assertion of the lemma is true for $m=0$ by keeping in mind that $r_0'=0$. 
Now let $0<m \leq \kappa$. By assumption, we have
\[
  x^+=T^+y=(T^*T)^\mu w\,,
\]
and we choose a  positive $\epsilon$ such that
\begin{equation} \label{eq1:errlemma}
0<\epsilon \leq |r_m'(0)|^{-1}\,,
\end{equation}
which in particular implies that $\epsilon$ is smaller than or equal to $\lambda_{1,m}$, cf. $(\ref{rd0})$. Next, we introduce
\begin{equation} \label{iterreprat}
  x_m^\delta=g_m(T^*T)T^*y^\delta, \quad g_m(\lambda)=\frac{q_{m-1}(\lambda)}{(1+\lambda\slash \gamma)^{m-1}} \in \Pi_{m-1}\slash(1+\cdot\slash \gamma)^{m-1}\,, 
\end{equation}
where $g_m$ is the rational function that represents the $m$-th SINE-iterate in $\mathcal{Q}_m$.
We obtain
\begin{eqnarray*}
 \|x^+-x_m^\delta\| 
 &\leq
 \|E_\epsilon(x^+-x_m^\delta)\| + \|(I-E_\epsilon)(x^+-x_m^\delta)\| \\
 &\leq
 \|E_\epsilon(x^+-g_m(T^*T)T^*y)\| + \|E_\epsilon(g_m(T^*T)T^*y-x_m^\delta)\| \\
 & \hspace{8cm}  
   +\epsilon^{-\frac{1}{2}}\|y-Tx_m^\delta \| \\
 &\leq 
 \|E_\epsilon r_m(T^*T)(T^*T)^\mu w\| + \|E_\epsilon g_m(T^*T)T^*(y-y^\delta)\| \\
 & \hspace{8cm} 
   +\epsilon^{-\frac{1}{2}}\|y-Tx_m^\delta \| \\
 &\leq
 \|\lambda^\mu r_m(\lambda)\|_{C[0,\epsilon]}\rho + \|\lambda^{\frac{1}{2}}g_m(\lambda)\|_{C[0,\epsilon]}\delta
   +\epsilon^{-\frac{1}{2}}(\|y^\delta-Tx_m^\delta \|+\delta)\,. \\
\end{eqnarray*}
From here, a literal copy of the proof of Lemma~7.11 in \cite{EHaNeu96} will do. The only difference is that $r_m$ is a rational function (cf. $(\ref{rdetail})$) instead of a polynomial.
\end{proof}
Finally, we can prove our main theorem.
\begin{theorem}\label{sineoptreg}
If $y \in \mathcal{R}(T)$ and if {\sc SINE} is stopped according to the discrepancy principle $(\ref{discrep})$ with $m(\delta,y^\delta)$, then {\sc SINE} is an order-optimal regularisation method, i.e., 
if $T^+y \in \mathcal{X}_{\mu,\rho}$, then
\[
  \|T^+y-x_{m(\delta,y^\delta)}^\delta\| \leq c\rho^{\frac{1}{2\mu+1}}\delta^{\frac{2\mu}{2\mu+1}}\,.
\]  
\end{theorem}
\begin{proof}
By the definition of the stopping index $m(\delta,y^\delta)$ by the discrepancy principle, one obtains
\[
  \delta_{m(\delta,y^\delta)} = \max\{\|y^\delta-Tx_{m(\delta,y^\delta)}^\delta\|,\delta \} \leq \tau \delta\,.
\]
With respect to lemma~\ref{errlem}, it remains to estimate $|r_{m(\delta,y^\delta)}'(0)|$.
For simplicity, we write $m$ instead of $m(\delta,y^\delta)$ in the following and assume, without loss of generality, that $m \geq 2$. ($m=0$ follows from lemma~\ref{errlem} with $r_0'=0$, $m=1$ refers to the space $\mathcal{Q}_1=\mathcal{K}_1$ and Theorem~7.12 in \cite{EHaNeu96} applies).  
By lemma~\ref{ressetlem}, we conclude that 
\[
 \tau \delta < \|y^\delta-Tx_{m-1}^\delta\| \leq \delta + c|r_{m-1}'(0)|^{-\mu-\frac{1}{2}} \rho\,.
\]
Since $\tau > 1$, this implies that
\begin{equation} \label{eq1:thm}
|r_{m-1}'(0)| \leq c\left( \frac{\rho}{\delta} \right)^{\frac{2}{2\mu+1}}\,.
\end{equation}
It remains to estimate
\[
\pi_m := r_{m-1}'(0)-r_m'(0)\,.
\]
The rational function
\[
  u_m(\lambda) := \frac{r_{m-1}(\lambda)-r_m(\lambda)}{\lambda} \in \Pi_{m-1}\slash(1+\cdot \slash \gamma)^{m-1}
\]
satisfies
\[
[u_m,\lambda \varphi]=0 \qquad \mbox{for every}\quad\varphi \in \Pi_{m-2}\slash(1+\cdot \slash \gamma)^{m-2}
\]
due to $(\ref{rperpQm})$ in lemma~\ref{rperpq}.
Moreover, by definition of $\pi_m$ and $(\ref{rd0})$, 
\begin{equation} \label{eq2:thm}
u_m(0)=\pi_m > \frac{1}{\gamma} >0\,.
\end{equation}
Substituting $u_m=\pi_m+\lambda\varphi$, then we have $\varphi \in \Pi_{m-2}\slash(1+\cdot \slash \gamma)^{m-1}$
and
\begin{equation} \label{eq8:thm}
 [u_m,u_m]=\pi_m[u_m,1]+[u_m,\lambda\varphi]\,.
\end{equation}
We first show that
\begin{equation} \label{pimab:1}
 [u_m,\lambda\varphi] = [u_m,u_m-\pi_m] 
                      = -\frac{1}{\gamma}[r_{m-1},\frac{r_{m-1}}{\lambda}]
                        +\frac{1}{\gamma}[r_{m},\frac{r_{m}}{\lambda}]\,.
\end{equation}
Using $(\ref{iterreprat})$, we obtain
\begin{eqnarray*}
u_m(\lambda)&=\frac{r_{m-1}(\lambda)-r_m(\lambda)}{\lambda}
  = \frac{1}{\lambda}
    \left(
   1-\lambda \frac{q_{m-2}(\lambda)}{(1+\lambda \slash \gamma)^{m-2}} \right. \\
     & \hspace{7cm} \left. -
      \left(  
         1-\lambda \frac{q_{m-1}(\lambda)}{(1+\lambda \slash \gamma)^{m-1}}
      \right)
    \right)  \\
  &= \frac{q_{m-1}(\lambda)}{(1+\lambda \slash \gamma)^{m-1}} 
      - \frac{q_{m-2}(\lambda)}{(1+\lambda \slash \gamma)^{m-2}}
\end{eqnarray*}
and
\[
  \pi_m=u_m(0)=q_{m-1}(0)-q_{m-2}(0)\,.
\]
Hence,
\begin{eqnarray*}
[u_m,u_m-\pi_m]&=
  \left[
   \frac{r_{m-1}-r_m}{\lambda},
     \frac{q_{m-1}}{(1+\lambda \slash \gamma)^{m-1}}-\frac{q_{m-2}}{(1+\lambda \slash \gamma)^{m-2}} \right. \\
    & \hspace{6.5cm}  
     +q_{m-2}(0)-q_{m-1}(0)
  \Big] \\
   &=
  \left[
    \frac{r_{m-1}-r_m}{\lambda}, 
     \lambda \frac{\tilde{q}_{m-2}}{(1+\lambda \slash \gamma)^{m-1}}
       -\lambda \frac{\tilde{q}_{m-3}}{(1+\lambda \slash \gamma)^{m-2}}
  \right]\,,
\end{eqnarray*}
with
\begin{eqnarray*}
  \tilde{q}_{m-2}(\lambda)&=\frac{q_{m-1}(\lambda)-q_{m-1}(0)(1+\lambda \slash \gamma)^{m-1}}{\lambda} \in \Pi_{m-2}\,, \\
  \tilde{q}_{m-3}(\lambda)&=\frac{q_{m-2}(\lambda)-q_{m-2}(0)(1+\lambda \slash \gamma)^{m-2}}{\lambda} \quad \left\{ 
  \begin{array}{ccc}
   \in \Pi_{m-3} &\mbox{for}& m \geq 3 \\ 
       =  0   &\mbox{for}& m=2
  \end{array}
\right. \,. \\
\end{eqnarray*}
Since 
\[
\left[
  r_m, \frac{\tilde{q}_{m-2}}{(1+\lambda \slash \gamma)^{m-1}}
       -\frac{\tilde{q}_{m-3}}{(1+\lambda \slash \gamma)^{m-2}}
\right]=0
\]
and
\[
\left[
  r_{m-1}, \frac{\tilde{q}_{m-3}}{(1+\lambda \slash \gamma)^{m-2}}
\right]=0
\]
according to lemma~\ref{rperpq}, we have, again with lemma~\ref{rperpq},
\begin{eqnarray*}
[u_m,u_m-\pi_m]&= 
 \left[
  r_{m-1}, \frac{\tilde{q}_{m-2}}{(1+\lambda \slash \gamma)^{m-1}} 
 \right] \\ 
 &=
 \left[
  r_{m-1}, \frac{\tilde{q}_{m-2}}{(1+\lambda \slash \gamma)^{m-1}} 
  -
  \frac{\tilde{q}_{m-2}}{(1+\lambda \slash \gamma)^{m-2}}
 \right] 
 \\
 &=
 \left[
  r_{m-1}, -\frac{\lambda}{\gamma}\cdot\frac{\tilde{q}_{m-2}}{(1+\lambda \slash \gamma)^{m-1}} 
 \right] \\
 &= 
 \left[
  r_{m-1}, -\frac{\lambda}{\gamma}\cdot \frac{q_{m-1}-q_{m-1}(0)(1+\lambda \slash \gamma)^{m-1}}{\lambda\,(1+\lambda \slash \gamma)^{m-1}} \right] \\ 
 &=
 -\frac{1}{\gamma}\cdot
 \left[
  r_{m-1}, \frac{q_{m-1}}{(1+\lambda \slash \gamma)^{m-1}}-q_{m-1}(0) 
 \right] \\
 &= 
 -\frac{1}{\gamma}\cdot
 \left[
  r_{m-1}, \frac{q_{m-1}}{(1+\lambda \slash \gamma)^{m-1}} 
 \right] \\
 &=
 -\frac{1}{\gamma}\cdot
 \left[
  r_{m-1}, \frac{1-r_m}{\lambda} 
 \right]
 =
 -\frac{1}{\gamma}\cdot
 \left[
  r_{m-1}, \frac{1}{\lambda} 
 \right]
 +
 \frac{1}{\gamma}\cdot
 \left[
  r_{m-1}, \frac{r_m}{\lambda} 
 \right]\,.
\end{eqnarray*}
By lemma~\ref{rperpq} and $(\ref{iterreprat})$, we obtain
\[
 [r_{m-1},\frac{1}{\lambda}] = [r_{m-1},-g_{m-1}+\frac{1}{\lambda}]=[r_{m-1},\frac{1}{\lambda}r_{m-1}]
\]
and, similarly,
\[
 [r_{m-1},\frac{r_m}{\lambda}]=[1-\lambda \frac{q_{m-2}}{(1+\lambda\slash \gamma)^{m-2}},\frac{r_m}{\lambda}]
 = [1,\frac{r_m}{\lambda}]=[\frac{1}{\lambda},r_m]=[\frac{r_m}{\lambda},r_m]
\]
which finally gives $(\ref{pimab:1})$. 
Due to
\[
  [u_m,1]=[r_{m-1},\frac{1}{\lambda}]-[r_m,\frac{1}{\lambda}]
\]
and by lemma~\ref{rperpq}, we obtain analogously
\[
 [r_j,\frac{1}{\lambda}]=[r_j,-g_j+\frac{1}{\lambda}]=[r_j,\frac{1}{\lambda}r_j], \qquad j=m-1,m,
\]
and therefore
\begin{equation} \label{eq9:thm}
  [u_m,1]=[r_{m-1},\frac{1}{\lambda}r_{m-1}]-[r_m,\frac{1}{\lambda}r_m]\,.
\end{equation}
Hence, by setting $(\ref{pimab:1})$ and $(\ref{eq9:thm})$ in $(\ref{eq8:thm})$, we obtain
\begin{equation} \label{eq3:thm}
[u_m,u_m]=\left(\pi_m-\frac{1}{\gamma}\right)[r_{m-1},\frac{1}{\lambda}r_{m-1}]-\left(\pi_m-\frac{1}{\gamma}\right)[r_m,\frac{1}{\lambda}r_m]\,. 
\end{equation}
From here, the proof continues literally as the proof of Theorem~7.12 in \cite{EHaNeu96}.
\end{proof}

\section{Upper bounds for the stopping index} \label{sec:stop}
The number $m(\delta,y^\delta)$ of necessary iterations to meet the discrepancy principle reflects the efficiency of the method. Due to construction, 
SINE will stop faster with respect to the discrepancy principle than any other method that relies on the shift-and-invert Krylov subspace. Here, we will additionally show, that SINE stops earlier than CGNE (or at the same iterate as CGNE, in the worst case) under the same assumptions as in section~\ref{sec:sine}. For this discussion, we designate the stopping index for SINE with parameter $\gamma>0$ as 
$m^\gamma(\delta,y^\delta)$. When $\gamma$ tends to infinity, SINE turns into CGNE. Therefore we designate the stopping index of CGNE with $m^\infty(\delta,y^\delta)$. 
\begin{theorem} \label{SINEfasterthanCGNE} 
If $y \in \mathcal{R}(T)$ and $\gamma >0$ then
$
   0 \leq m^\gamma(\delta,y^\delta) \leq m^\infty(\delta,y^\delta) < \infty\,. 
$
\end{theorem}
\begin{proof}
For $m\geq 1$, let $p_m \in \Pi_m$ be an arbitrary polynomial of degree less than or equal to $m$. With $r_m(\lambda)=p_m(\lambda)\slash(1+\lambda\slash \gamma)^{m-1}$, we have
\[
  \int_0^{\|T\|^2+} r_m^2(\lambda) d\|F_\lambda y^\delta\|^2 \leq \int_0^{\|T\|^2+} p_m^2(\lambda) d\|F_\lambda y^\delta\|^2
\]
and therefore
\[
  \min_{\stackrel{r \in \Pi_m\slash(1+\cdot\slash \gamma)^{m-1}}{r(0)=1}} \int_0^{\|T\|^2+} r_m^2(\lambda) d\|F_\lambda y^\delta\|^2 \leq \int_0^{\|T\|^2+} p_m^2(\lambda) d\|F_\lambda y^\delta\|^2
\]
for all $p_m \in \mathcal{P}_m$ with $p_m(0)=1$, and, finally,
\[
  \min_{\stackrel{r \in \Pi_m\slash(1+\cdot\slash \gamma)^{m-1}}{r(0)=1}} \int_0^{\|T\|^2+} r_m^2(\lambda) d\|F_\lambda y^\delta\|^2 \leq \min_{\stackrel{p \in \Pi_m}{p(0)=1}} \int_0^{\|T\|^2+} p_m^2(\lambda) d\|F_\lambda y^\delta\|^2\,.
\]
Hence, by $(\ref{resrep})$, we have 
\[
  \min_{x \in \mathcal{Q}_m} \|y^\delta-Tx\| \leq \min_{x \in \mathcal{K}_m} \|y^\delta-Tx\|\,.
\]
Due to the definition of CGNE $(\ref{minpropCGNE})$ and SINE $(\ref{minprop})$, it follows immediately that the norm of the residual of the $m$th-SINE iterate $x_m^{\mbox{\tiny SINE}}$ is always smaller than or equal to the norm of the residual of the $m$th-CGNE iterate $x_m^{\mbox{\tiny CGNE}}$, i.e.,
\[
    \|y^\delta-Tx_m^{\mbox{\tiny SINE}}\| \leq \|y^\delta-Tx_m^{\mbox{\tiny CGNE}}\|
\]
holds for all $m \geq 0$, which proves our theorem. (The case $m=0$ is trivial.)
\end{proof}
Theorem~\ref{SINEfasterthanCGNE} basically shows that all upper bounds that are known for CGNE also apply to SINE, but SINE might be faster. We explicitly state some corollaries. As a corollary of theorem~7.13 in \cite{EHaNeu96}, which is due to \cite{NemPoI,NemPoII}, and theorem~\ref{SINEfasterthanCGNE} we obtain the following statement.
\begin{corollary}
If $y \in \mathcal{R}(T)$, $\gamma \in \mathbb{R}^+ \cup \{\infty\}$, and $T^+y \in \mathcal{X}_{\mu,\rho}$, then
\[
   m^\gamma(\delta,y^\delta) \leq c\left(\frac{\rho}{\delta}\right)^{\frac{1}{2\mu+1}}\,
\]
and this estimate is sharp in the sense that the exponent cannot be replaced by a smaller one and that the bound is supposed to hold true for all possible values of $\gamma$.
\end{corollary}
Theorem~7.14 and theorem~7.15 in \cite{EHaNeu96} also hold literally for SINE as simple corollaries of theorem~\ref{SINEfasterthanCGNE}

\section{Illustration and discussion} \label{sec:numexp}
We use the multiplication operator $T: L_2(0,1) \rightarrow L_2(0,1)$, $Tf(t):=tf(t)$, in order to illustrate the theoretical findings. The range of $T$ is not closed. For example, it can be readily seen that any constant function apart from zero is in the closure $\overline{\mathcal{R}(T)}$ of the range $\mathcal{R}(T)$, but not in the range $\mathcal{R}(T)$ itself. We further have $T^*=T$ and $T^*Tf(t)=t^2f(t)$. For the fractional powers of the operator $T^*T$, we obtain $(T^*T)^\mu f(t)=t^{2\mu}f(t)$, $\mu > 0$. We choose the exact solutions $x^+_1=t$ and $x^+_2=t^3$ with right-hand sides
$y_1=Tx_1^+=t^2$ and $y_2=Tx_2^+=t^4$, respectively. Then $x_1^+ \in \mathcal{X}_{1\slash 2,1}$ and $x_2^+ \in \mathcal{X}_{3\slash 2,1}$. 
We use the perturbed right-hand sides $y_i^\delta=y_i+\delta$, $i=1,2$, with $\|y_i^\delta-y_i\|=\delta$. The linear systems with the perturbed right-hand sides $y_i^\delta$, $i=1,2$, do not have a solution and a regularisation is necessary. 
We now use SINE to compute regularised solutions, where the iteration is stopped according to the discrepancy principle with $\tau=1001\slash 1000$.
In figure~\ref{ratecheckcont}, the $L_2$-norm of the error of the computed regularisation is plotted versus $\delta^{-1}$. 
The red circle-marked line belongs to the error with respect to $x_1^+=t$ and the green, square-marked line belongs to the error with respect to $x_2^+=t^3$. As predicted by theorem~\ref{sineoptreg}, the convergence to the exact solution with decreasing perturbation $\delta$ is at least $\delta^{\frac{1}{2}}$ or $\delta^{\frac{3}{2}}$, respectively, which are indicated by the gray lines. The operator is simple enough such that all computations could be conducted exactly by using the computer algebra system Maple.
\begin{figure}
\begin{center}
\begin{tikzpicture}

\begin{axis}[
width=14cm, height=7cm, scale only axis,
%xmin=170, xmax=3300, ymode=log, xmode=log, ymin=3e-5, ymax=3e-2,
xmin=9, xmax=11000, ymode=log, xmode=log, ymin=2e-4, ymax=2e-1,
xtick={10, 100, 1000,10000}, 
ytick={1,1e-1,1e-2,1e-3,1e-4,1e-5,1e-6}, yminorticks=false
]
\addplot[color=red, mark=*, thick, solid] coordinates {
(10,   0.115690616)
%(40,   0.046593895)
(100,  0.031844382)
%(300,  0.017391047)
%(400,  0.0149053823)
(1000, 0.0091457408)
%(5000, 0.0037182479)
(10000,0.0023973761)
};
\addplot[gray, thick, domain=10:10000]
  { 0.5*x^(-1/2) };
\addplot[color=green, mark=square*, thick, solid] coordinates {
(10,   0.0743658494)
(100,  0.0114035821)
(1000, 0.001901342909)
%(5000, 0.0005035690)
(10000,0.00026552051)
};
\addplot[gray, thick, domain=10:10000]
  { 0.5*x^(-3/4) };
\end{axis}
\end{tikzpicture}%
\end{center}
\caption{$L_2$-error versus inverse noise $\delta^{-1}$}
\label{ratecheckcont}
\end{figure}
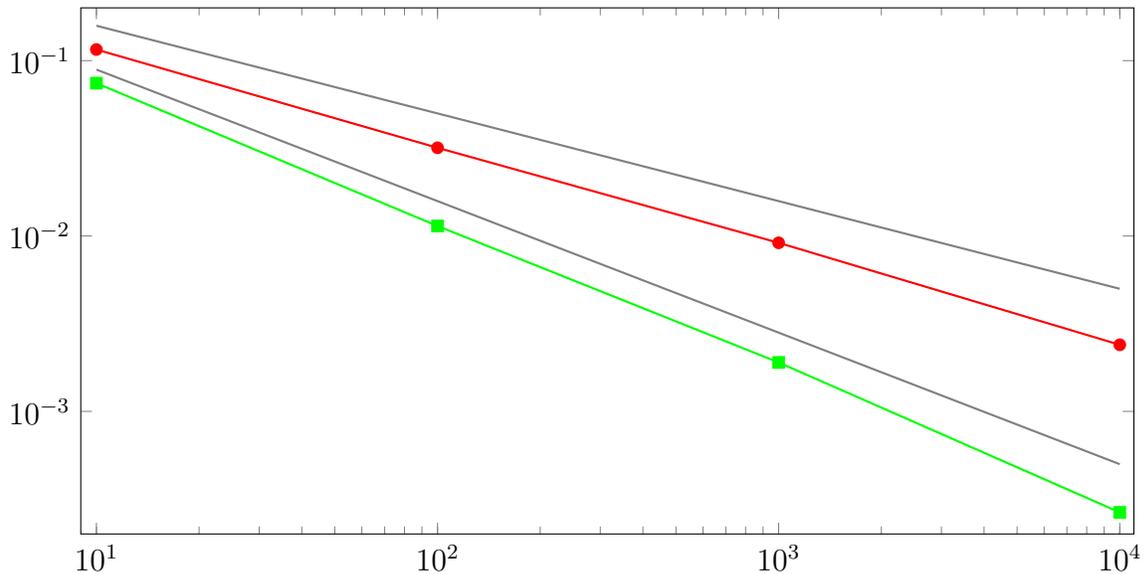
Due to construction, SINE will stop faster with respect to the discrepancy principle than any other method computing regularisations in the shift-and-invert Krylov subspace $\mathcal{Q}_m$. 
That is, where these methods have been used successfully, SINE should also be a very good choice.
That SINE might also be useful with respect to regularisation schemes that do not use the shift-and-invert Krylov subspace $\mathcal{Q}_m$, will be illustrated by another simple experiment, where we compare SINE and CGNE.
For $\gamma=\frac{1}{1000}$ and $x^+=t$, $y^\delta=t^2+\delta$, $\delta=\frac{1}{1000}$, SINE stops after two steps with
\[
x_2=-\frac{21}{5000}t^3+\frac{1507}{1500}t=c_1T^*y^\delta+c_2 (1+T^*T\slash \gamma)^{-1}T^*y^\delta \in \mathcal{Q}_2\,,
\]
$c_1=-21\slash 5000$, $c_2=15070063\slash 15000$, 
whereas CGNE produces a polynomial of degree $39$ after $19$ steps. Both methods have been stopped according to the discrepancy principle with $\tau=\frac{1001}{1000}$. In figure~\ref{CGNEvsSINE}, it can be seen that the SINE regularisation on the left-hand side is qualitatively better than the CGNE regularisation on the right-hand side.  
The experiment also shows that the stopping index of SINE can be significantly smaller than the stopping index of CGNE. With the designations of section~\ref{sec:stop}, we have
\[
  m^\gamma(\delta,y^\delta)=2 < 19=m^\infty(\delta,y^\delta).
\]

\begin{figure}
\begin{center}
\begin{center}
\begin{tikzpicture}
\begin{axis}[
width=6cm, height=6cm, scale only axis,
xmin=0, xmax=1, ymin=0, ymax=1
]
\addplot[red, thick, domain=0:1]
  { -21/5000*x^3+1507/1500*x};
\end{axis}
\end{tikzpicture}
\hspace*{1cm}
\begin{tikzpicture}
\begin{axis}[
width=6cm, height=6cm, scale only axis,
xmin=0.0, xmax=1, ymin=0.0, ymax=1
]
\addplot[red, thick, domain=0:0.4]
  {  
     -3.333889334*10^9*x^37 
     + 3.307942782*10^8*x^39 
     + 1.555859473*10^10*x^35 
     -1.139752527*10^11*x^25
     -4.461177603*10^10*x^33
     -1.263317845*10^11*x^29
     - 3.741947655*10^10*x^21
     + 8.793822957*10^10*x^31 
     +1.481132113*10^10*x^19 + 1.076851218*10^9*x^15-1.923674702*10^8*x^13+7.386323446*10^10*x^23
     -2.369285579*10^6*x^9+149163.9396*x^7-5796.0929*x^5-4.555254788*10^9*x^17+1.36817676*10^11*x^27
     +2.532581405*10^7*x^11+118.875264*x^3+0.1248229156*x
  };
\addplot[red, thick, domain=0.4:0.9]
   {
 0.2789678843e-1+0.4598231983e-1*(x-.5000000000)^2+17.60189170*(x-.5000000000)^3-23.87329759*(x-.5000000000)^4-1894.925027*(x-.5000000000)^5+1973.595399*(x-.5000000000)^6+94551.45465*(x-.5000000000)^7-45881.24465*(x-.5000000000)^8-2667338.658*(x-.5000000000)^9-452701.7354*(x-.5000000000)^10+46749297.65*(x-.5000000000)^11+40679024.90*(x-.5000000000)^12-525792077.3*(x-.5000000000)^13-860354826.2*(x-.5000000000)^14+3683368333.*(x-.5000000000)^15+9770363508.*(x-.5000000000)^16-0.1333428914e11*(x-.5000000000)^17-0.6660746580e11*(x-.5000000000)^18-7376246144.*(x-.5000000000)^19+0.2662976910e12*(x-.5000000000)^20+0.3197823055e12*(x-.5000000000)^21-0.4853866128e12*(x-.5000000000)^22-0.1442678527e13*(x-.5000000000)^23-0.4719845036e12*(x-.5000000000)^24+0.2562638869e13*(x-.5000000000)^25+0.3912492765e13*(x-.5000000000)^26+0.3782516997e12*(x-.5000000000)^27-0.5257116663e13*(x-.5000000000)^28-0.6823429659e13*(x-.5000000000)^29-0.2465431763e13*(x-.5000000000)^30+0.3510831466e13*(x-.5000000000)^31+0.6328691293e13*(x-.5000000000)^32+0.5371450065e13*(x-.5000000000)^33+0.2986020430e13*(x-.5000000000)^34+0.1160976032e13*(x-.5000000000)^35+0.3162141609e12*(x-.5000000000)^36+0.5794575070e11*(x-.5000000000)^37+6450488424.*(x-.5000000000)^38+330794278.2*(x-.5000000000)^39+.9484194288*x
   };
\addplot[red, thick, domain=0.9:1]
  {
-1.5959068449052952056+340.08793846041800437*(x-1.)^2+31906.407124184181297*(x-1.)^3+1669964.4440696963231*(x-1.)^4+55412628.629165945543*(x-1.)^5+1263298808.5102602327*(x-1.)^6+20908582776.767939478*(x-1.)^7+261467899679.56600112*(x-1.)^8+2546219965719.4283168*(x-1.)^9+19767834892368.044623*(x-1.)^10+124660291735918.40881*(x-1.)^11+648301151477175.55025*(x-1.)^12+2815101749000044.4631*(x-1.)^13+10311852697782690.589*(x-1.)^14+32137302787866479.391*(x-1.)^15+85821651306958571.006*(x-1.)^16+197542493437015194.37*(x-1.)^17+393831875492617749.19*(x-1.)^18+682747936969055633.92*(x-1.)^19+1032433623248708481.8*(x-1.)^20+1365010256663444170.9*(x-1.)^21+1580462869598788458.3*(x-1.)^22+1603980221970814019.9*(x-1.)^23+1427111346790524937.2*(x-1.)^24+1112529656923535935.3*(x-1.)^25+758855602200783313.27*(x-1.)^26+451864821740558898.53*(x-1.)^27+234109567094136738.53*(x-1.)^28+105055412977416624.94*(x-1.)^29+40586490794111979.521*(x-1.)^30+13392312169354097.630*(x-1.)^31+3735052407457255.7466*(x-1.)^32+868283051271217.92000*(x-1.)^33+165097351910729.14065*(x-1.)^34+25003348472654.613493*(x-1.)^35+2899775002907.5847031*(x-1.)^36+241784670795.76588119*(x-1.)^37+12900976848.938306332*(x-1.)^38+330794278.17790529057*(x-1.)^39+2.5987772118873411219*x    
  };
\end{axis}
\end{tikzpicture}
\end{center}
\end{center}
\caption{Left-hand side SINE-regularisation attained in 2nd step, right-hand side CGNE-regularisation attained in 19th step}
\label{CGNEvsSINE}
\end{figure}
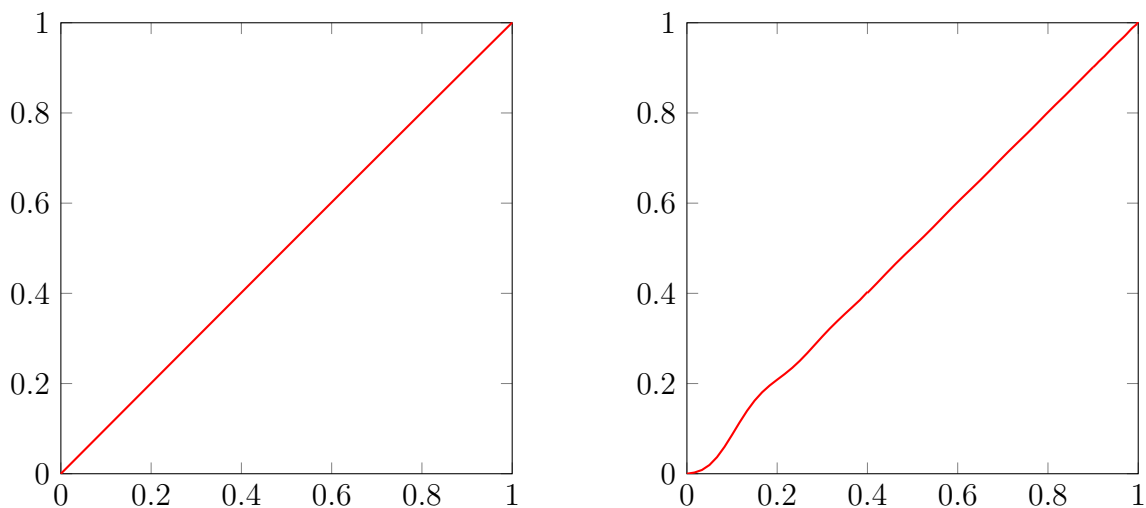

Altogether, the theory and the experiment suggest that SINE is a valid order-optimal regularisation scheme.
It is also hoped, that the given analysis inspires further research on the regularisation properties of rational Krylov subspace methods. For example, it is immediately clear that theorem~\ref{SINEfasterthanCGNE} carries over to rational Krylov subspaces with arbitrary negative real poles, when the method is defined analogous to~$(\ref{minprop})$.
Even choosing negative poles at random can only improve on CGNE with respect to the stopping index. 
These more general rational Krylov subspace methods might also be seen as accelerations of the nonstationary iterated Tikhonov iteration (e.g. \cite{JinStal12}) or of method ${\sl (ii)}$ in example~\ref{methodswithqm} with varying step sizes. While there is only one polynomial Krylov subspace, rational Krylov subspaces inspire a wide range of methods that might be adapted to the needs at hand. 
As a possible application, rational Krylov subspaces have been successfully used to accelerate computations related to seismic imaging (e.g. \cite{DruReZasZimm18,Liuetal19,Zhouetal18,ZimmerDiss18}), which is known to be an ill-posed inverse problem. 
\section*{Acknowledgments}
\emph{
Funded by the Deutsche Forschungsgemeinschaft (DFG, German Research Foundation) -- Project-ID 258734477 -- SFB 1173
}

\bibliographystyle{plain}
\bibliography{lit}

\end{document}